\documentclass[12pt]{article}
\usepackage{hyperref,graphicx,fancyhdr,fancybox,rotating,xcolor,appendix}
\usepackage{ulem,hhline,dsfont,pifont,multicol,lmodern}
\usepackage{amscd}
\usepackage{lipsum}
\usepackage[all,cmtip]{xy}
\usepackage{amsmath,amsthm,amsfonts,amsbsy,amssymb,geometry,times,pst-node}
\usepackage{mathrsfs} 
\usepackage{latexsym}
\usepackage{mathtools}
\usepackage{enumitem} 
\usepackage[utf8]{inputenc}
\usepackage[english]{babel}
\usepackage{ mathrsfs }

\newcommand{\Z}{\mathbb{Z}}
\newcommand{\Q}{\mathbb{Q}}

\newcommand{\tens}[1]{%
  \mathbin{\mathop{\otimes}\displaylimits_{#1}}}
\usepackage[all]{xy}
\def\Mod{{\rm Mod}}

\def\Ext{{\rm Ext}}

\def\Hom{{\rm Hom}}
\def\Im{{\rm Im}}
\def\Ker{{\rm Ker}}
\def\Coker{{\rm Coker}}

\def\wdim{{\rm wdim}}

\newtheorem{thm}{\bf Theorem}[section]
\newtheorem{cor}[thm]{\bf Corollary}
\newtheorem{lem}[thm]{\bf Lemma}
\newtheorem{prop}[thm]{\bf Proposition}
\newtheorem{Def}[thm]{\bf Definition}

\newtheorem{ex}[thm]{\bf Example}

\def\I{{\mathcal{I}}}
\def\PI{{\mathcal{PI}}}

\def\p{{\mathcal{P}}}
\def\f{{\mathcal{F}}}

\def\x{{\mathcal{X}}}
\def\A{{\mathscr{A}}}
\def\y{{\mathcal{Y}}}

\def\ML{{\mathcal{ML}}}

\def\X{{ {\mathcal{X}}}^{-1}}
\def\Y{{ {\mathcal{Y}}}^{-1}}

\def\F{{ {\mathfrak{F}}}^{-1}}
\def\P{{\underline {\mathfrak{Pr}}}^{-1}}
\def\In{{\underline {\mathfrak{In}}}^{-1}}

\def\A{{ {\mathcal{A}p}}}
\def\+{^+}
\begin{document}
	\title{Flat-precover completing domains}
	\author{Houda Amzil, Driss  Bennis, J. R. Garc\'{\i}a Rozas and Luis Oyonarte}

	\date{}
	
	\maketitle

	\bigskip
	
	\noindent{\large\bf Abstract.} Recently, many authors have embraced the study of certain properties of modules such as projectivity, injectivity and flatness from an alternative point of view. Rather than saying a module has a certain property or not, each module is assigned a relative domain which, somehow, measures to which extent it has this particular property. In this work, we introduce a new and fresh perspective on flatness of modules. However, we will first investigate a more general context by introducing domains relative to a precovering class $\x$. We call these domains $\x$-precover completing domains. In particular, when $\x$ is the class of flat modules, we call them flat-precover completing domains. This approach allows us to provide a common frame for a
	number of classical notions. Moreover, some known results are generalized and some classical rings are characterized in terms of these domains.

	
	\bigskip
	

	
	
	
	\section{Introduction}
	
	Throughout this paper, $R$ will denote an associative ring with identity and modules will be unital left $R$-modules, unless otherwise explicitely stated.
	As usual, we denote by $R$-$\Mod$ the category of left $R$-modules. We denote the class of projective modules by $\p$, the class of injective modules by $\I$ and the class of flat modules by $\f$. For a module $M$, $E(M)$, $PE(M)$, and $M^+$ denote its injective envelope, its pure-injective preenvelope and its character module $\Hom_{\Z}(M,\Q/\Z)$.

	To any given class of modules $\mathcal{L}$ we associate its right Ext-orthogonal class, $$\mathcal{L}^{\perp}= \{M\in R\text{-}\Mod \mid \Ext^1(L,M) = 0, L\in \mathcal{L}\},$$ and its left Ext-orthogonal class, $${}^{\perp}\mathcal{L}= \{M\in R\text{-}\Mod \mid \Ext^1(M,L) = 0, L\in \mathcal{L}\}.$$ In particular, if $\mathcal{L}= \{M\}$ then we simply write ${}^{\perp}\mathcal{L}={}^{\perp} M$ and  $\mathcal{L}^{\perp}=M^{\perp}$.

	Recall that given a class of modules $\x$, an $\x$-precover of a module $M$ is a morphism $X\rightarrow M$ with $X\in \x$, such that $\Hom(X',X)\rightarrow \Hom (X',M)\rightarrow 0$ is exact for any $X'\in\x$. An $\x$-precover
	$\phi :X\to M$ is said to be an $\x$-cover if every endomorphism $g :X\to X$ such that $\phi g=\phi$ is
	an isomorphism. Dually we have the definitions of an $\x$-preenvelope and an $\x$-envelope. If every module has an $\x$-(pre)cover, $\x$ is said to be (pre)covering. Similarly, if every object has an $\x$-(pre)envelope, we say that $\x$ is (pre)enveloping. 
	
	Relative notions of projectivity, injectivity and flatness were introduced as a tool to evaluate the extent of these properties for any given module. In contrast to the well-known notion of relative injectivity, Aydoğdu and L\'opez-Permouth introduced in \cite{PinSergio} the notion of subinjectivity. Then Holston et al. introduced in \cite{Sergio} the projective analogue of subinjectivity and called it subprojectivity. Namely, a module $M$ is said to be $N$-subprojective if for
	every epimorphism $g: B\to N$ and homomorphism $f: M \to N$ , then there exists a
	homomorphism $h: M \to B$ such that $gh = f$. For a module $M$, the subprojectivity
	domain of $M$, $\P(M)$ , is defined to be the collection of all modules $N$ such that $M$ is $N$-subprojective.
	
	On the other hand, the study of flatness was accessed in \cite{Rugged} and \cite{Dur} from two slightly similar alternative perspectives as both use the tensor product. Indeed, relative flatness studied in \cite{Rugged} is defined as follows: Given a right R-module $N$, a left $R$-module $M$ is said to be flat relative to $N$, relatively flat to $N$, or $N$-flat if the canonical morphism $K\tens{R}M \to N\tens{R}M$ is a monomorphism for every submodule $K$ of $N$. The flat domain of a module $M$, $\f^{-1}(M)$, is defined to be the collection of right modules $N$ such that $M$ is $N$-flat. From the definition, it is clear that a module $M$ is flat if and only if $\f^{-1}(M)=\Mod\text{-}R$. Durğun in \cite{Dur}, modifies in a subtle way the notion of relative flatness domains and defines absolutely pure domains. Namely, given a left module $M$ and a right module $N$, $N$ is said to be absolutely $M$-pure if $N\tens{R}M \to B\tens{R}M$ is a monomorphism for every extension $B$ of $N$. For a module $M$, the absolutely
	pure domain of $M$, $\A(M)$, is defined to be the collection of all modules $N$ such that $N$ is
	absolutely $M$-pure. Clearly, a module $M$ is flat if and only if $\A(M)=\Mod\text{-}R$.
	
    In this paper, we define a new alternative perspective on flatness of modules, inspired by similar ideas studied in several papers about subprojectivity domains. In this process, projective modules should in general be replaced by flat modules. Thus in this work, we follow a pattern which has been somewhat established in previous studies. However, the nature of projectivity and flatness are curiously different that each domain can be remarkably unique. This will be shown throughout this paper.

	The paper is organized as follows:
	
	We start first by investigating a general context by introducing domains relative to a precovering class $\x$. 
	To this end, we define the $\x$-precover completing domain $\X(M)$ for a module $M$ (see Definition \ref{Def-subd}). When $\x$ is the class of projective modules, $\x$-precover completing domains will be nothing but subprojectivity domains. Moreover, it is easy to show that a module $M\in \x$ if and only if its $\x$-precover completing domain consists of the entire class of modules $R$-$\Mod$. And if $N\in \x$, then $M$ is vacuously ($N$,$\x$)-precover completing. Section 2 is devoted to the basic properties of $\x$-precover completing domains of modules.  We extend the study done in \cite{SubprojAb} and \cite{Sergio} to the relative case and gather different closeness properties that $\x$-precover completing domains verify (see Propositions \ref{closure-prop-X},\ref{prop-clos-kern} and \ref{prop-0ABX0}). Then, we investigate  the $\x$-precover completing domain of a module embedded in a module in the class $\x$. Namely, we show that given a short exact sequence of the form $0\to M \to X\to M'\to 0$, we have ${M'}^{\perp}\subseteq\X(M)$ (Proposition \ref{M-X-M'}) and as a consequence, $\X(M)$ contains the class of injective modules. This brings us to Propositions \ref{prop-Inj-X(M)} and \ref{prop-PureInj-X(M)} where we establish equivalent conditions for $\x$-precover completing domains to contain the particular class of injective modules $\I$ and that of pure-injective modules $\PI$, respectively. Finally, in Proposition \ref{prop-X-Y} we compare the domains relative to two precovering classes $\x$ and $\y$ such that $\x\subseteq \y$.\\
	 In Section 3, we shed light on flat-precover completing domains. This leads to new characterizations of known notions. For instance, for any ring $R$, (1) $\wdim(R)\leq 1$ if and only if the flat-precover completing domain of any module is closed under submodules; (2) $R$ is right coherent if and only if the flat-precover completing domain of any module is closed under direct products (Proposition \ref{closure-prop}). This approach also allows us to give interesting examples. Indeed, we show that if $M$ is Ding projective then there exists a Ding projective module $M'$ such that $\F(M)=M'^{\perp}$, and if $M$ is a finitely presented and strongly Gorenstein flat module then $\F(M)=M^\perp$ (see Example \ref{cor-SGF1}). Then, as in the classical homological context where the relation between flat and projective modules is extensively studied, we investigate the relation between flat-precover completing domains and subprojectivity domain (see Proposition \ref{prop-flat-proj-dom}). This shows us a new side to well-known notions (see Corollary \ref{cor-coinc-subflat-subproj} and Proposition \ref{prop-shenglin}). In Proposition \ref{prop-Flatness-AbsPure}, we provide a new perspective on the known-result proved in \cite{Lambek}: A module is flat if and only if its character module is absolutely pure. Finally in \cite{CS}, coherent rings are characterized by the equivalence of the absolutely purity of modules and the flatness of their character modules, here we show the counterpart result in our context. Namely, we characterize right coherent rings by means of flat-precover completing domains and absolutely pure domains (see Proposition \ref{prop-Flatness-AbsPure2}).\\
	 Finally, in Section 4 we focus on the relation between flat-precover completing domains and subinjectivity domains. We start first with two results (Proposition \ref{prop-Flatness-Subinj} and \ref{prop-QF}) which relates flat-precover completing domains and subinjectivity domains. Then we consider the following question: What is the structure of a ring over which the flat-precover completing domains and subinjectivity domains coincide? We prove that a ring satisfies this condition if and only if every factor ring of $R$ is $QF$ if and only if $R$ is isomorphic to a product of full matrix rings over Artinian chain rings (see Theorem \ref{thm-coinc-subflat-subinj}). Then, we give equivalent characterizations for flat-precover completing domains to contain the class of injective modules (Proposition \ref{prop-embedds}). As a consequence, this result allows us to give a straightforward proof (see Corollary \ref{cor-IF-rings}) to characterizations of $IF$-rings established by Colby in \cite[Theorem 1]{Colby}.
	 In \cite{PinSergio}, a module with the smallest possible subinjectivity domain is said to be indigent.  Here we introduce the opposite concept to that of flat modules; that is, modules for which the flat-precover completing domains are as small as possible. We call these new modules f-rugged modules. We show that there are f-rugged modules for any arbitrary ring and finally, we establish a connection between f-rugged modules and indigent modules (Proposition \ref{prop-frugged}).
	
\section{Relative domains: Basic results}\label{section-subX}
In this section, we give general results which most of them can be proved as in \cite{SubprojAb}, so we will omit the proofs. \\
In what follows, $\mathcal{X}$ will denote a precovering class of modules which satisfies the following conditions: $\x$ is closed under isomorphisms, i.e., if $M\in \x$ and $N \cong M$, then $N \in \mathcal{X}$. We also assume that $\x$ is closed under taking finite direct sums and direct summands, i.e., if $M_1,...,M_t \in \x$ then $M_1 \oplus \cdots \oplus M_t \in \mathcal{X}$; if $M=N\oplus L \in \mathcal{X}$ then $N,L \in \mathcal{X}$.

\begin{Def}\label{Def-subd}
	
	Given modules $M$ and $N$, $M$ is said to be ($N$,$\x$)-precover completing if for every morphism $f:M\to N$, and every $\mathcal{X}$-precover $g:{X}\to N$, there exists a morphism $h:M\to {X}$ such that $gh=f$. When no confusion arises, we will omit the name of the class and say simply that $M$ is $N$-precover completing.
	The \textit{$\x$-precover completing domain} of a module $M$ is defined to be the collection 
	$$\X(M):=\{ N \in R\text {-}\Mod: M \ \text {is } \ N\text{-precover completing}\}.$$
\end{Def}

Then, if we take $\x$ to be the class of projective modules, the $\x$-precover completing domains are simply subprojectivity domains as defined in \cite{SubprojAb} and \cite{Sergio}. 

As the subprojectivity domain does, the $\x$-precover completing domain measures when a module belongs to the class $\x$.  This can be seen from the following result which follows simply from the definition.

\begin{prop}\label{prop-MinX}
	Let $M$ be a module. Then the following conditions are equivalent:
	\begin{enumerate}
		\item $M\in \x$.
		\item  $\X(M)$ is the whole category $R$-$\Mod$.
		\item $M\in\X(M)$.
	\end{enumerate} 
\end{prop} 

As in \cite[Proposition 2.7]{SubprojAb}, the following gives a simple characterization of the notion of $\x$-precover completing domains. 

\begin{lem}\label{Lemma-subX}
	Let $M$ and $N$ be two modules. Then the following assertions are equivalent:
	\begin{enumerate}
		\item $M$ is ($N$,$\x$)-precover completing.
		\item There exists an ${\mathcal{X}}$-precover $g:X\to N$ such that $\Hom(M,g)$ is an epimorphism.
		\item Every morphism $M\to N$ factors through a module in $\x$.
	\end{enumerate}
\end{lem}

As a consequence, it is clear that $\x\subseteq\X(M)$ for any module $M$.

In the following result, we gather different closeness properties that $\x$-precover completing domains verify. We omit
its proof, which has much in common with proofs in \cite[Section 3]{SubprojAb}. Recall that $M$ is said to be a small module if  $\Hom(M,-)$ preserves direct sums.

\begin{prop} \label{closure-prop-X}
	The following properties hold for any module $M$.
	\begin{enumerate}
		\item Given a short exact sequence of modules   $0\rightarrow A \rightarrow B \rightarrow C \rightarrow 0$ which is $\Hom(\x,-)$ exact, if   $A$ and $C$ are in $ \X(M)$, then   $B$ is in  $\X(M)$.
		\item For a finite family of modules $\{N_i; 1,...,m\}$, $N_i\in \X(M)$ for every $i\in \{1,...,m\}$, if and only if  $\oplus_{i=1}^{m}N_i \in \X(M)$.
		\item 	If $M$ is a small module then the $\x$-precover completing domain of $M$ is closed under arbitrary direct sums; that is, if $N_i\in \X(M)$ for every $i\in I$, then $M$ is $\oplus_{i\in I}N_i \in \X(M)$. 
		\item Let $\{M_i; i\in I\}$be a family of modules. Then, $\X(\oplus_{i\in I}M_i)=\bigcap_{i\in I}\X(M_i)$.
		\item  The class $\x$ is closed under arbitrary direct products if and only if the $\x$-precover completing domain of any module is closed under arbitrary direct products. 
		\item The class $\x$ is closed under submodules if and only if the $\x$-precover completing domain of any module is closed under submodules. 
		
	\end{enumerate} 
\end{prop}

It is not known whether or not the subprojectivity domains are closed under kernels of epimorphisms. In \cite[Proposition 3.2]{SubprojAb}, a weak equivalent condition for this property was provided. In the relative case we need to assume the further condition that the short exact sequences are $\Hom(\x,-)$ exact.

\begin{prop} \label{prop-clos-kern}
	Suppose that $\x$ contains the class of projective modules and let $M$ be a module. Then the following conditions are equivalent:
	\begin{enumerate}
		\item For every short exact sequence $0\to A\to B\to C\to 0$ which is $\Hom(\x,-)$ exact, if $B,C\in \X(M)$ then $A\in \X(M)$. 
		\item For every short exact sequence $0\to K\to X \to C \to 0$ where $X\to C$ is an $\x$-precover, 
		if $C \in\X(M)$ then $K\in\X(M)$.
		\item For every $\x$-precover $X\to C$ with $C \in\X(M)$, the pullback of $X$ over $C$ holds in $\X(M)$.
	\end{enumerate}
\end{prop}

\begin{proof}

$ (1)\Rightarrow (2)$ Clear. \\
$ (2)\Rightarrow (1)$ Consider an exact sequence $$0\to A\to B\to C\to 0$$ which is $\Hom(\x,-)$ exact with $B,\ C\in \X(M)$. We have the following pullback diagram $$
\xymatrix{
	&&  0 \ar[d]&  0 \ar[d]&\\
	&&  K \ar@{=}[r] \ar[d]&  K \ar[d]\\
	0\ar[r]& A \ar[r] \ar@{=}[d]&  D \ar[r] \ar[d]&  X \ar[r] \ar[d]&  0\\
	0\ar[r]& A \ar[r] &  B \ar[r] \ar[d] &  C \ar[r] \ar[d] &  0\\
	&&0&0&
}
$$ where $X\to C$ is an $\x$-precover. Then, $K\in \X(M)$ by assumption because $C\in \X(M)$. $D$ being the pullback of $X\to C$ and $B\to C$ and since $X\to C$ is an $\x$-precover, we can easily see that the sequence $0\to K\to D\to B\to 0$ is $\Hom(\x,-)$ exact. Then, by assertion 1 in Proposition \ref{closure-prop-X}, $D \in\X(M)$. And since $0\to A\to B\to C\to 0$ is $\Hom(\x,-)$ exact and $D$ is the pullback of $X\to C$ and $B\to C$, we see that $0\to A \to D \to X\to 0$ splits. Thus, $A$ is a direct summand of $D$. Using assertion 2 in Proposition \ref{closure-prop-X} we deduce that $A \in\X(M)$.\\
$ (2)\Leftrightarrow (3)$ Consider the following diagram where $D$ is the pullback of $X$ over $C$ $$\xymatrix{
	0\ar[r]& K \ar[r] \ar@{=}[d]&  D \ar[r] \ar[d]&  X \ar[r] \ar[d]&  0\\
	0\ar[r]& K \ar[r] &  X \ar[r]&  C \ar[r] &  0
}$$

Then, the short exact sequence $0\to K\to D\to X\to 0$ is $\Hom(\x,-)$ exact.
If $C \in \X(M)$ then $K \in \X(M)$ and so by assertion 1 in Proposition \ref{closure-prop-X}, $D \in\X(M)$. Conversely, if $D \in\X(M)$ then, by assertion 2 in Proposition \ref{closure-prop-X} $K \in\X(M)$ (because the short exact sequence $0\to K\to D\to X\to 0$ splits).
\end{proof}

It is easy to see that if the $\x$-precover completing domain of every module is closed under kernels of epimorphisms, then $\x$ is also closed under kernels of epimorphisms. For the converse, the following result shows that we have a partial positive answer when we consider some special epimorphisms.

\begin{prop}\label{prop-0ABX0}
	The class $\x$ is closed under kernels of epimorphisms if and only if for every module $M$ and every short exact sequence $0\to A \to B \to X \to 0 $ with $X\in \x$ , if $B\in \X(M)$ then $A\in \X(M)$. 
\end{prop}

\begin{proof}
	We first show that $\x$ is closed under kernels of epimorphisms. Consider the short exact sequence $0\to K \to X_1\to X_2 \to 0$ with $X_1,X_2\in \x$. Clearly, we have $X_1,X_2 \in \X(K)$ so $K\in \X(K)$ and by Proposition \ref{prop-MinX} we deduce that $K\in \x$.\\
	For the converse, let $X'\to B$ be an $\x$-precover and consider the following pullback diagram
	
	$$\xymatrix{0\ar[r]& D \ar[r]\ar[d]& X' \ar[r]\ar[d]&X\ar[r]\ar@{=}[d]&0\\
		0\ar[r]& A \ar[r]&B\ar[r]&X\ar[r]&0 \\}$$ 
	Since $\x$ is closed under kernels of epimorphisms, we have $D\in \x$. Applying the functor $\Hom(M,-)$, we obtain
	$$\xymatrix{0\ar[r]& \Hom(M,D) \ar[r]\ar[d]& \Hom(M,X') \ar[r]\ar[d]&\Hom(M,X)\ar@{=}[d]\\
		0\ar[r]& \Hom(M,A) \ar[r]&\Hom(M,B)\ar[r]&\Hom(M,X) \\}$$ 
	
	Since $B\in \X(M)$, $\Hom(M,X') \to \Hom(M,B)$ is epic and so $\Hom(M,D) \to \Hom(M,A)$ is epic too. Therefore, $A\in \X(M)$.
\end{proof}

In the next proposition, we investigate the $\x$-precover completing domain of a module that can be embedded in a module in the class $\x$.

\begin{prop}\label{M-X-M'}
	Let $0\to M \stackrel{\alpha}\to X\to M'\to 0$ be a short exact sequence with $X\in \x$. Then, ${M'}^{\perp}\subseteq\X(M)$. Moreover, if $\alpha:M\to X$ is an $\x$-preenvelope, then $\X(M)\cap X^{\perp} \subseteq {M'}^{\perp}$.
\end{prop}

\begin{proof}
	Let $N\in {M'}^{\perp}$. Consider the following long exact sequence $$\xymatrix{\ar[r]& \Hom(X,N) \ar[r]& \Hom(M,N) \ar[r]& \Ext^1(M',N)\ar[r]& \Ext^1(X,N) \ar[r]&}$$
	
	Since $\Ext^1(M',N)=0$, $\Hom(X,N) \to \Hom(M,N)$ is epic. Thus, $N\in \X(M)$ by Lemma \ref{Lemma-subX}. 
	
	Now, let $N\in\X(M)\cap X^{\perp}$ and $X'\to N$ be an $\x$-precover. We obtain the following commutative diagram by applying the functors $\Hom(-,X')$ and $\Hom(-,N)$ to $M\to X$  $$
	\xymatrix{
		\Hom(X,X')\ar[r]\ar[d]& \Hom(M,X')\ar[d] & \\
		\Hom(X,N) \ar[r]& \Hom(M,N)\ar[r] & \Ext^1(M',N)\\
	}
	$$
	
	Since $N\in X^{\perp}$, $\Ext^1(X,N)=0$. From the long exact sequence above, we can see that we only need to prove that $\Hom(X,N)\to \Hom(M,N)$ is epic to deduce that $\Ext^1(M',N)=0$. Since $N\in\X(M)$ we have $\Hom(M,X') \to \Hom(M,N)$ is an epimorphism. Moreover $\Hom(X,X')\to\Hom(M,X')$ is epic because $\alpha:M\to X$ is an $\x$-preenvelope. Therefore, $\Hom(X,N)\to \Hom(M,N)$ is epic. This completes the proof.
\end{proof}

From Proposition \ref{M-X-M'}, we see that if a module $M$ embeds in a module in $\x$ then $\X(M)$ contains the class of injective modules.  In the next result, we prove that this is in fact an equivalence and also establish some other equivalences similarly to \cite[Lemma 2.2]{Dur0}
\begin{prop}\label{prop-Inj-X(M)}
	The following conditions are equivalent for a module $M$:
	\begin{enumerate}
		\item $M$ embeds in a module in $\x$.
		\item There exists a module $M'$ such that $M'^\perp\subseteq \X(M)$.
		\item $\I\subseteq \X(M)$.
		\item $E(M)\in \X(M)$.
		\item For any flat right $R$-module $F$, $F^+ \in  \X(M)$.
	\end{enumerate}
\end{prop}

\begin{proof}
	$(1)\Rightarrow (2)$ Follows from Proposition \ref{M-X-M'}.\\
	$(2) \Rightarrow (3) \Rightarrow (4)$ Clear.\\
	$(4)\Rightarrow (1)$ Since $E(M)\in \X(M)$, the inclusion map $f:M\to E(M)$ factors through a module in $\x$; that is,  there exist two morphisms $h:M\to X$ and $g:X\to N$ with $X\in \x$ such that $f=gh$. But $gh$ is a monomorphism and so $h:M\to X$ is monomorphism. We conclude that $M$ can be embedded in a module in $\x$.\\
	$(3)\Rightarrow (5)$ Clear.\\
	$(5)\Rightarrow (3)$ Let $E$ be an injective module and let $F\to  E^+\to 0$ be a flat precover. Then we have $0\to E^{++}\to F^+$. And since there exists an injective map $E\to E^{++}$, we have a map $0\to E\to F^+$ which splits. We have $F^+\in \X(M)$ and so by assertion 2 in Proposition \ref{closure-prop-X}, we obtain $E \in \X(M)$.
\end{proof}

Now we investigate when the domains contain the class of pure-injective modules $\PI$.

\begin{prop}\label{prop-PureInj-X(M)}
	The following conditions are equivalent for a module $M$:
	\begin{enumerate}
		\item $M$ is a pure submodule of a module in $\x$.
		\item $\PI\subseteq \X(M)$.
		\item $PE(M)\in \X(M)$.
		\item For any right $R$-module $N$, $N^+ \in  \X(M)$.
		\item For any pure-projective right $R$-module $P$, $P^+ \in  \X(M)$.
	\end{enumerate}
\end{prop}

\begin{proof}
	
	$(1)\Rightarrow (2)$ Let $E$ be a pure-injective module. Let $f:M\to E$ be a morphism. We have $M$ is a pure submodule of a module $X\in\x$. We denote by $i:M\to X$ the pure monomorphism. As $E$ is pure-injective, there exists a morphism $h:X\to E$ such that $f=hi$ and so $f:M\to E$ factors through a module in $\x$. Thus, $E\in \X(M)$.\\
	$(2)\Rightarrow (3)$ Clear.\\
	$(3)\Rightarrow (1)$ Since $PE(M)\in \X(M)$, the inclusion map $f:M\to PE(M)$ factors through a module in $\x$; that is,  there exist two morphisms $h:M\to X$ and $g:X\to N$ with $X\in \x$ such that $f=gh$. But $gh$ is a pure monomorphism and so $h:M\to X$ is a pure monomorphism. \\
	$(2)\Rightarrow (4)$ Clear because for every module $N$, $N^+$ is pure-injective.\\
	$(4)\Rightarrow (5)$ Clear.\\
	$(5)\Rightarrow (2)$ Let $E$ be a pure-injective module and let $P\to  E^+\to 0$ be a pure epimorphism with $P$ pure-projective. Then we have $0\to E^{++}\to P^+$ splits. And since there exists a pure monomorphism $E\to E^{++}$, we have a map $0\to E\to P^+$ which splits. Then, $E$ is a direct summand of $P^+$ and, by assumption, $P^+\in \X(M)$. Thus, by assertion 2 in Proposition \ref{closure-prop-X}, we obtain $E \in \X(M)$.
\end{proof}

In the next proposition, we establish a connection between domains relative to two precovering classes $\mathcal{X}$ and $\mathcal{Y}$ such that $\x\subseteq \y$.

\begin{prop}\label{prop-X-Y}
	Let $\mathcal{X}$ and $\mathcal{Y}$ be two precovering classes such that $\x\subseteq \y$ and let $M$ be a module. Then, $\X(M)\subseteq \Y(M)$. Furthermore, $\X(M)= \Y(M)$ if and only if $\y \subseteq \X(M)$.
\end{prop}

\begin{proof}
	Let $M$ be a module and let $N\in \X(M)$. Then any morphism $M\to N$ factors through a module in $\x$. Since $\x\subseteq \y$, we deduce that any morphism $M\to N$ factors through a module in $\y$. Thus, $N\in \Y(M)$. \\
	Suppose now that $\y \subseteq \X(M)$. Let $N\in \Y(M)$ and $f:M\to N$ be a morphism. Then there exists $g:Y\to N$ and $h: M\to Y$ such that $Y\in \y$ and $gh=f$. Since $\y \subseteq \X(M)$, $h:M\to Y$ factors through a module in $\x$. Hence,  $f:M\to N$ factors through a module in $\x$. Therefore, $N\in \X(M)$. Finally, it is clear that if $\X(M)= \Y(M)$ then $\y \subseteq \X(M)$.
\end{proof}

\section{Flat-precover completing domains}

Now we focus our attention on the aim of this paper, which is when the class $\x$ is that of flat modules. In that case, we refer to $\x$-precover completing domains by flat-precover completing domains and  we denote the flat-precover completing domain of a module $M$ by $\F(M)$.  Flat-precover completing domains measure flatness of modules. Indeed, from Lemma \ref{Lemma-subX} we can see that $N$ is flat if and only if $N \in \F(M)$ for any module $M$ if and only
$N \in \F(M)$ for any finitely presented module $M$. \\
Clearly, all the properties proved in Section \ref{section-subX} remain valid for flat-precover completing domains so we will omit repeating them here.  
The purpose of this section is twofold, to prove the utility of flat-precover completing domains by characterizing some classical rings in terms of these newly defined domains and to study their relationship with subprojectivity domains and absolutely pure domains (see \cite{Dur} and \cite{Sergio}). \\
First, let us begin with the following extension of the characterizations of a flat module in terms of flat-precover completing domains. Notice that, following \cite{Sergio}, a module is projective if and only if $\P(M)=R$-$\Mod$. Thus, $M$ is projective if and only if $\p^\perp\subseteq \P(M)$. In the following proposition we show that the analogue result also holds for flat-precover completing domains. 

Recall that a  module $C$ is called cotorsion if $\Ext^1(F,C)=0$ for any flat module $F$. We denote the class of cotorsion modules by $\mathcal{C}$. A monomorphism $\alpha:M \to C$ with $C$ cotorsion is said to be a special cotorsion preenvelope of
M if $\Coker{\alpha}$ is flat. Recall, from \cite[Section 7.4]{EJ}, that every module $M$ has a special cotorsion preenvelope that we denote by $C(M)$.

\begin{prop}\label{prop-M-flat}
	Let $M$ be a module. Then the following conditions are equivalent:
	\begin{enumerate}
		\item $M$ is flat.
		\item  $\F(M)=R\text{-}\Mod$.
		\item $M\in\F(M)$.
		\item $\mathcal{C}\subseteq \F(M)$.
		\item $C(M) \in \F(M)$.
		\item $\PI \subseteq \F(M)$.
		\item $PE(M)\in \F(M)$.
		\item For any module $N$, $N^+\in \F(M)$.\item For any pure-projective module $P$, $P^+\in \F(M)$.

	\end{enumerate} 
\end{prop}

\begin{proof}
$(1) \Leftrightarrow (2) \Leftrightarrow (3)$ Follow from Proposition \ref{prop-MinX}.\\
$(2) \Rightarrow (4) \Rightarrow (5)$ Clear. \\
$(5) \Rightarrow (3)$ Consider the short exact sequence $0\to M \stackrel{\alpha}\to C(M)\to \Coker\alpha\to 0$ where $\alpha$ is a special preenvelope (so $\Coker(\alpha)$ is flat). Since $C(M) \in \F(M)$, Proposition \ref{prop-0ABX0} implies that $M\in \F(M)$.\\
$(1) \Leftrightarrow (6) \Leftrightarrow (7) \Leftrightarrow (8) \Leftrightarrow (9)$ Follow from Proposition \ref{prop-PureInj-X(M)}.
\end{proof}

From Proposition \ref{prop-M-flat} we see that if $P$ is a projective module then $M$ is flat if and only if $\F(M)=P^\perp$. But there are other examples of $M$ such that $\F(M)=N^\perp$ for some module $N$ (without $N^\perp$ being the whole category of modules).
For, recall
 a module $M$ is said to be strongly Gorenstein flat if there is an exact sequence $\cdots \to P_1 \to P_0 \to P^0 \to P^1 \to \cdots$ of projective modules
with $M=\Ker(P_0\to P^0) $ such that $\Hom(-,\f)$ leaves the sequence exact (see \cite[Definition 2.1]{Ding}). Gillespie in \cite{Gillespie} renamed these modules as Ding projective modules since there exists an alternative definition of a strongly Gorenstein flat module. This definition was given in \cite[Definition 3.1]{Bennis} and thus we understand by a strongly Gorenstein flat module, a module $M$ such that there exists an exact sequence of flat modules
$$\cdots \stackrel{f}\to F\stackrel{f}\to F\stackrel{f}\to F \stackrel{f}\to \cdots$$
such that $M=\Im(F \stackrel{f}\to F)$ and such that $I\tens{R}-$ leaves the sequence exact whenever $I$ is an injective right module.  

\begin{ex}\label{cor-SGF1}
	\begin{enumerate}
\item For any Ding projective module $M$, there exists a Ding projective module $M'$ such that $\F(M)=M'^{\perp}$.
\item  For any strongly Gorenstein flat and finitely presented module M, $\F(M)=M^\perp$.
\end{enumerate}
	
\end{ex}

\begin{proof}
	
	1. Any Ding projective module $M$ has an $\f$-preenvelope
	$\alpha: M \to F$ with $F$ projective. Moreover, $\Coker(\alpha)$ is Ding projective. We apply Proposition \ref{M-X-M'} to deduce that $\F(M)={\Coker{(\alpha)}}^{\perp}$.\\
	2. By \cite[Proposition 3.9]{Bennis}, a module is finitely generated and strongly Gorenstein projective if and only if it is finitely presented and strongly Gorenstein flat. By Corollary \ref{cor-f(FP)=pr(FP)}, $\F(M)=\P(M)$ for any finitely presented and strongly Gorenstein flat module $M$ and by \cite[Corollary 2.9]{SubprojAb} we can conclude that $\F(M)=M^\perp$. 
\end{proof}

In the classical homological algebra, relations between flat modules, projective modules and absolutely pure modules have been extensively studied. In our new context, we investigate the counterpart results. Namely, we study the relation between flat-precover completing domains, subprojectivity domains and  absolutely
pure domains (see \cite{SubprojAb} and \cite{Dur}). We start by investigating the relationship between flat-precover completing domains and subprojectivity domains. \\
Clearly, the subprojectivity domain $\P(M)$ of a module $M$ is contained in $\F(M)$. However, they are not necessary equal as it is shown by the following example.

\begin{ex}
The abelian group $\Q$ is a flat $\Z$-module, thus by Proposition \ref{prop-MinX}, $\Q\in\F(\Q)$. But $\Q\notin\P(\Q)$ for otherwise $\Q$ would be a projective $\Z$-module.
\end{ex}
  
 From Proposition \ref{prop-X-Y} we deduce the following result. 

\begin{prop}\label{prop-flat-proj-dom}
	Let $M$ be a module. Then, $\F(M)=\P(M)$ if and only if $\f\subseteq \P(M)$.
\end{prop}

Recall that a ring $R$ is perfect if and only if any flat module is projective. In terms of flat-precover completing domains, we have the following result. 
\begin{cor}\label{cor-coinc-subflat-subproj}
	The following conditions are equivalent:
	\begin{enumerate}
\item $R$ is perfect.
\item $\F(M)=\P(M)$ for any module $M$.
\item $\F(M)\subseteq \P(M)$ for any module $M$.
	\end{enumerate}
\end{cor}

It is shown in \cite[Corollary 2.19]{SubprojAb} that the subprojectivity domain of any pure-projective module contains the class of flat modules. Thus, by Proposition \ref{prop-flat-proj-dom}, we deduce the following result.

\begin{cor}\label{cor-f(FP)=pr(FP)}
	For any pure-projective module $M$, $\F(M)=\P(M)$.
\end{cor}

Corollary \ref{cor-f(FP)=pr(FP)} stands as a generalization to the very well-known result that any finitely presented and flat module is projective. Indeed, if $N$ is a finitely presented and flat module, then $N\in \F(N)$ and so $N\in\P(N)$. Therefore, $N$ is projective.\\

We recall that a module $M$ is said to be f-projective if, for
every finitely generated submodule $C$ of $M$, the
inclusion map factors through a finitely generated free module. And we say that a module $M$ is Mittag-Leffler if for every finitely generated submodule $C$ of $M$, the inclusion map factors through a finitely presented module (see \cite{Jones}). We denote the class of Mittag-Leffler modules by $\ML$.\\ We apply Proposition \ref{prop-flat-proj-dom} to characterize when the flat-precover completing domains and subprojectivity domains of finitely generated modules coincide. We obtain a characterization of rings over which every flat module is f-projective in terms of flat-precover completing domains. An extensive study of rings over which every flat module is f-projective is done by Shenglin (see \cite[Corollary 5]{Shenglin}).

\begin{prop}\label{prop-shenglin}
	The following assertions are equivalent:
	\begin{enumerate}
	    \item Any flat module is f-projective.
		\item Any flat module is Mittag-Leffler.
		\item $\F(M)=\P(M)$ for any finitely generated module $M$.
		\item The class of modules holding in $\F(M)$ for any finitely generated module $M$ is precisely the class of f-projective modules.
	\end{enumerate} 
\end{prop}

\begin{proof}
	$(1) \Leftrightarrow (2)$ By \cite[Proposition 1]{Jones}, a module $M$ is f-projective
	if and only if $M$ is flat and Mittag-Leffler.\\
	$(1) \Leftrightarrow (3) $ By Proposition \ref{prop-flat-proj-dom} $\F(M)=\P(M)$ for every of finitely generated modules $M$ if and only if the class of flat modules is inside the subprojectivity domain of every finitely generated module. By \cite[Proposition 2.1]{SubprojAb} the class of modules holding in the subprojectivity domain of every finitely generated module is precisely the class of f-projective modules. Thus, $\F(M)=\P(M)$ for every finitely generated modules $M$ if and only any flat module is f-projective. \\
	$(3) \Rightarrow (4) $ Follows from \cite[Proposition 2.1]{SubprojAb}.\\
	$(4) \Rightarrow (1) $ Clear.
	\end{proof}

Lambek \cite{Lambek} proved that, over any ring, a module is flat if and
only if its character module is absolutely pure. The next proposition generalizes this fact while connecting flat-precover completing domains with absolutely pure domains.

Let $M$ be a left $R$-module and $N$ be a right $R$-module. Recall that $N$ is absolutely
$M$-pure if $N\otimes_{R} M\to K\otimes_{R} M$ is a monomorphism for every extension $K$ of $N$, or equivalently if there exists an absolutely pure extension $E$ of $N$ such that $N\otimes_{R} M\to E\otimes_{R} M$
is a monomorphism (see \cite[Proposition 2.2]{Dur}). For a module $M$, the absolutely
pure domain of $M$, $\A(M)$, is defined to be the collection of all modules $N$ such that $N$ is
absolutely $M$-pure.

\begin{prop}\label{prop-Flatness-AbsPure}
	Let $M$ be a finitely presented module and $N$ a module. Then, $N\in \F(M)$ if and only if $N^+ \in \A(M)$. 
	
\end{prop}

\begin{proof}
	Let $F\to N\to 0$ be a flat precover of $N$. Since $M$ is finitely presented, by \cite[Theorem 3.2.11]{EJ} we have the following commutative diagram 
	$$
	\xymatrix{
		(\Hom(M,N))^+\ar[r]\ar[d]_\cong& (\Hom(M,F))^+\ar[d]^\cong\\
		N^+\otimes_{R} M \ar[r] & F^+\otimes_{R} M\\
	}
	$$
	
	Hence,  $N^+\otimes_{R} M \to F^+\otimes_{R} M$ is monic if and only if $\Hom(M,F) \to \Hom(M,N)$ is epic.  Therefore, $N^+ \in \A(M)$ if and only if $N\in \F(M)$.
	\end{proof}

We end this section with some characterizations of coherent rings in terms of flat-precover completing domains. First notice that from Proposition \ref{closure-prop-X} we get immediately the following result and recall that a ring $R$ is left semihereditary if and only if $R$ is left coherent and $\wdim R\leq 1$. 

\begin{prop} \label{closure-prop}
	The following properties hold:
	\begin{enumerate}
		\item The ring $R$ is right coherent if and only if the flat-precover completing domain of any module is closed under arbitrary direct products. 
		\item $\wdim(R) \leq 1$ if and only if the flat-precover completing domain of any module is closed under submodules.
		\item The ring $R$ is left semihereditary if and only if the flat-precover completing domain of any right $R$-module is closed under arbitrary direct products and the flat-precover completing domain of any module is closed under submodules.
		
	\end{enumerate} 
\end{prop}

Cheatham and Stone in \cite[Theorem 1]{CS} characterized right coherent rings in terms of absolutely pure and flat modules. The next proposition gives the counterpart result in terms of flat-precover completing domains and absolutely pure domains.

\begin{prop}\label{prop-Flatness-AbsPure2}
	
	The following assertions are equivalent:
	\begin{enumerate}
		\item The ring $R$ is right coherent.
		\item For any module $M$ and any right module $N$, $N\in \A(M)$ if and only if $N^+\in\F(M)$.
		\item For any finitely presented module $M$ and any module $N$, $N\in \F(M)$ if and only if $N^{++}\in\F(M)$.
	\end{enumerate}
\end{prop}

\begin{proof}
	$(1)\Rightarrow (2)$ Let $0\to N \to E$ be an injective preenvelope of $N$. Since $R$ is coherent, $E^+\to N^+ \to 0$ is a flat precover by \cite[Proposition 5.3.5]{EJ}. We have the following commutative diagram 
	
	$$
	\xymatrix{
		\Hom(M,E^+)\ar[r]\ar[d]_\cong& \Hom(M,N^+)\ar[d]^\cong\\
		(E\otimes_{R} M)^+ \ar[r] & (N\otimes_{R} M)^+\\
	}
	$$
	
	Hence, 	$\Hom(M,E^+)\to \Hom(M,N^+) $ is epic if and only if
	$N\otimes_{R} M\to E\otimes_{R} M $ is monic. Therefore,  $N^+\in\F(M)$ if and only if $N\in \A(M)$.\\
	$(2)\Rightarrow (3)$ Let $M$ be a finitely presented module and $N$ be a module. Then $N\in \F(M)$ if and only if $N^{+}\in\A(M)$ by Proposition \ref{prop-Flatness-AbsPure} and so $N\in \F(M)$ if and only if $N^{++}\in\F(M)$ by assumption.\\
	$(3)\Rightarrow (1)$ Let $N$ be a flat module. Then by Proposition \ref{prop-M-flat}, $N\in \F(M)$ for any module $M$. By assumption $N^{++}\in\F(M)$ for any module $M$ and so again by Proposition \ref{prop-M-flat} $N^{++}$ is flat. We conclude by \cite[Theorem 1]{CS} that $R$ is right coherent.
 \end{proof}

\section{Subinjectivity and flat-precover completing domains}

The relation between flat and injective modules was widely studied by several authors throughout the years leading to a wide range of rich results. In this section, we study to what extent we can compare the degrees of flatness and injectivity of modules.

Recall first the definition of subinjectivity domains. We say that a module $M$ is $N$-subinjective if for every extension $K$ of $N$ and every morphism $f : N \to M$ there exists a morphism $g : K \to M$ such that $g_{/N}=f$. The subinjectivity domain of a module $M$, $\In(M)$, is defined to be the collection of all modules $N$ such that $M$ is $N$-subinjective (see \cite{PinSergio}).

We start with the following result.
\begin{prop}\label{prop-Flatness-Subinj} 
		The following properties hold:
	\begin{enumerate}
		\item For any module $M$ and any right module $N$, if $N^+\in \F(M)$ then $N\in \In(M^+)$.
		\item The ring $R$ is coherent if and only if for any module $M$ and any right module $N$, if $N\in \In(M^+)$ then $N^+\in \F(M)$.
	\end{enumerate}
\end{prop}

\begin{proof}
	For any module $L$ we denote by $\sigma_L:L\to L^{++}$ the evaluation morphism.\\
	$1.$ Let $N^+\in \F(M)$ with $F\to N^+$ a flat precover and let $f:N\to M^+$ be any morphism. Then we have a morphism $f^+:M^{++} \to N^+$.  Since $N^+\in \F(M)$, there exists a morphism $h:M\to F$ such that the following diagram commutes
	$$
	\xymatrix{
		& M \ar[d]^{\sigma_M} \ar@{.>}[ddl]_{\exists h} \\
		& M^{++}  \ar[d]^{f^+} \\
		F \ar[r]_{g} & N^+
	}
	$$
	
	Then, $gh=f^+\sigma_M$ and so $h^+g^+=\sigma_M^+f^{++}$. \\
	We have the following diagram with each square commutative 
	
	$$
	\xymatrix{
		N \ar[r]^{f} \ar[d]_{\sigma_N} & M^+ \ar[d]^{\sigma_{M^+}}  \\
		N^{++} \ar[r]^-{f^{++}} \ar[d]_{g^+} 	& M^{+++}  \ar[d]^{\sigma_M^+} \\
		F^+  \ar[r]_{h^+}   & M^+ \\
	}
	$$
	
	Now, $\sigma_M^+ {\sigma_{M^+}}=1_{M^+}$. Therefore, $f=\sigma_M^+ {\sigma_{M^+}}f=h^+g^+\sigma_N$. Since $F^+$ is injective, we deduce by \cite[Lemma 2.2]{PinSergio} that $N\in \In(M^+)$.\\
	$2.$  Suppose first that $R$ is coherent and let $N \in \In(M^+)$ with $g:N\to E$ an injective preenvelope (notice that $E^+$ is flat). Then for any morphism $f:M\to N^+$, there exists a morphism $h:E\to M^+$ such that the following diagram commutes
	$$
	\xymatrix{
		M^+ \\
		N^{++}  \ar[u]^{f^+} \\
		N \ar[r]_{g}\ar[u]^{\sigma_N} & E \ar@{.>}[uul]_{\exists h}
	}
	$$
	
	Then, $hg=f^+\sigma_N$ and so $g^+h^+=\sigma_N^+f^{++}$.
	We have the following diagram with each square commutative
	
	$$
	\xymatrix{
		M \ar[r]^{f} \ar[d]_{\sigma_M} & N^+ \ar[d]^{\sigma_{N^+}}  \\
		M^{++} \ar[r]^-{f^{++}} \ar[d]_{h^+} 	& N^{+++}  \ar[d]^{\sigma_N^+} \\
		E^+  \ar[r]_{g^+}   & N^+ \\
	}
	$$
	
	Again, $\sigma_N^+ {\sigma_{N^+}}=1_{N^+}$ so $f=g^+h^+\sigma_M$. Therefore, $N^+\in \F(M)$ by Lemma \ref{Lemma-subX}.\\
	Conversely now, to prove that $R$ is right coherent it suffices to prove by \cite[Theorem 1]{CS} that for any absolutely pure right module $N$, $N^+$ is flat. For let $N$ be an absolutely pure right module. Then $N\in \In(K)$ for any pure-injective module $K$. In particular, $N\in \In(M^+)$ for any module $M$ as every character module is pure-injective (see \cite[Proposition 5.3.7]{EJ}). By assumption, $N^+\in \F(M)$ for any module $M$ and so by Proposition \ref{prop-M-flat}, $N^+$ is flat. 
\end{proof}

One can see that the implication of the first assertion in Proposition \ref{prop-Flatness-Subinj} enables us to find again the well-known result: If a module $M$ is flat then $M^+$ is injective. Indeed, if $M$ is flat, then for any right $R$-module $N$, $N^+\in \F(M)$. Thus for any right $R$-module $N$, $N\in \In(M^+)$ and so $M^+$ is injective. One can also find again that for any module $N$, if $N^+$ is flat then $N$ is absolutely pure. Indeed, since $N^+$ is flat, we have $N^+\in \F(M)$ for any right $R$-module $M$ and so $N\in\In(M^+)$ for any right $R$-module $M$. Then we can easily show that $N$ is absolutely pure: Indeed, consider $0\to N\to E(N)$. Since $N\in \In(M^+)$, $\Hom(E(N),M^+)\to \Hom(N,M^+)$ is epic. But there is a natural isomorphism $\Hom(N,M^+)\cong (M\otimes N)^+$, then $M\otimes N\to M\otimes E(N)$ is monic for any module $M$ and so $N$ is absolutely pure. \\

One of the classical characterization of QF rings is that they are those rings for which any flat module is injective. Now we characterize QF rings in terms of flat-precover completing domains and subinjectivity domains.

\begin{prop}\label{prop-QF}
	
	The following assertions are equivalent:
	\begin{enumerate}
		\item The ring $R$ is QF.
		\item For any two modules $M$ and $N$, $N\in \F(M)$ if and only if $M\in \In(N)$.
	\end{enumerate}
\end{prop}

\begin{proof}
	$(1)\Rightarrow (2)$ Consider $0\to M\stackrel{i}\to E(M)$ and let $F\stackrel{g}\to N $ be a flat precover of $N$. Suppose that $N\in \F(M)$. To prove that $M\in \In(N)$, let $f:M\to N$ be any morphism. Since $N\in \F(M)$, there exists $h:M\to F$ such that $f=gh$. But $R$ is QF, so $F$ is injective. Thus there exists $l:E(M)\to F$ such that $li=h$ so  $f=gh=gli$ and by \cite[Lemma 2.2]{PinSergio} $M\in\In(N)$. Now suppose that $M\in\In(N)$ and let us prove that $N\in\F(M)$. For let $f:M\to N$ be any morphism. Since $M\in\In(N)$, there exists $h:E(M)\to N$ such that $hi=f$. But since $R$ is QF, $E(M)$ is flat and so by Lemma \ref{Lemma-subX} $N\in \F(M)$. 
	
	$(2)\Rightarrow (1)$  Let $F$ be a flat module. Then, $F\in \F(M)$, for every module $M$. By assumption we obtain that for every module $M$, $M\in \In(F)$ and so $F$ is injective. We conclude that $R$ is QF. 
\end{proof}

 In \cite{Sergio-Sim}, López-Permouth and Simental call a ring $R$ super QF-ring if the relative projectivity and relative injectivity domains coincide. In their paper, they prove that a ring $R$ is super $QF$-ring if and only if $R$ is isomorphic to a product of full matrix rings over artinian chain rings. Later, in \cite[Section 4]{Dur0} Y. Durğun proved that the subprojectivity and subinjectivity domains coincide if and only if $R$ is super $QF$-ring. 
In this section, we investigate the coincidence of flat-precover completing domains and subinjectivity domains. We will say that the flat-precover completing domains and subinjectivity domains coincide over $R$ if $\F(M)=\In(M)$ for any module $M$.

By following the same reasoning adopted in \cite{Dur0} and \cite{Sergio-Sim} , we show that a ring satisfies $\F(M)=\In(M)$ for any module $M$  if and only if $R$ is isomorphic to a product of full matrix rings over artinian chain rings.

\begin{prop}\label{factor-ring}
	Let $R$ be a ring over which $\F(M)\subseteq\In(M)$ for any module $M$ and let $I$ be an ideal of $R$. Then $\F(M')\subseteq\In(M')$ for any $R/I$-module $M'$.
\end{prop}

\begin{proof}
	We may identify $R/I$-$\Mod$ with the full subcategory of $R$-$\Mod$ consisting of modules which annihilated by $I$. For any $R/I$-module $M$, we have $\F({}_{R/I}M)=\F({}_RM)\cap R/I$-$\Mod$. But $\F({_R}M)\subseteq\In({_R}M)$ and $\In({}_RM)\cap R/I\text{-}\Mod=\In({_{R/I}}M)$. Thus $\F({}_{R/I}M)\subseteq \In({}_{R/I}M)$. 
\end{proof}

\begin{prop}\label{prop-prod-coinc}
	Let $R_1$ and $R_2$ be rings over which the flat-precover completing domains and subinjectivity domains coincide. Then, the flat-precover completing domains and subinjectivity domains coincide over $R_1 \times R_2$.
\end{prop}

\begin{proof}
	We have for any $R_1 \times R_2$-module $M$, $M = M_1\oplus M_2$ with $M_1\in R_1$-$\Mod$ and $M_2\in R_2$-$\Mod$. Thus $\F(M)=\F(M_1)\times \F(M_2)=\In(M_1)\times \In(M_2)=\In(M)$. We conclude that $R_1 \times R_2$ satisfies $\F(M)=\In(M)$ for any module $M$.
\end{proof}

\begin{prop}\label{prop-coinc-morita}
	Let $R$ be a ring over which the flat-precover completing domains and subinjectivity domains coincide and let $S$ be Morita equivalent to $R$. Then the flat-precover completing domains and subinjectivity domains coincide over $S$.
\end{prop}

\begin{proof}
	Since $S$ is Morita equivalent to $R$, there exists an equivalence of categories $\phi: R\text{-}\Mod \to S$-$\Mod$. It is easy to see that $A\in \F(M)$ if and only if $\phi(A) \in \F(\phi(M))$. Similarly, $A\in \In(M)$ if and only if $\phi(A) \in \In(\phi(M))$. Thus the result follows easily.
\end{proof}

\begin{thm}\label{thm-coinc-subflat-subinj}
	The following conditions are equivalent:
	\begin{enumerate}
		\item $\F(M)=\In(M)$ for any module $M$,
		\item $\F(M)\subseteq\In(M)$ for any module $M$,
		\item $\F(M)\subseteq\In(M)$ for any module over a factor ring of $R$, 
		\item Every factor ring of $R$ is QF,
		\item $R$ is isomorphic to a product of full matrix rings over artinian chain rings.
	\end{enumerate}
\end{thm}

\begin{proof}
	$(1)\Rightarrow (2)$ Clear. 
	
	$(2)\Rightarrow (3)$ Follows from Proposition \ref{factor-ring}.
	
	$(3)\Rightarrow (4)$ Follows from the fact that every flat module over a factor ring of $R$ is injective. Indeed, let $F$ be a flat $R/I$-module. Then $F\in \F(F)\subseteq \In(F)$ and so $F$ is injective. 
	
	$(4)\Rightarrow (5)$ Follows from \cite[Theorem 6.1]{Faith}.
	
	$(5)\Rightarrow (1)$ We suppose that $R\cong \prod_{i=1}^{k_i} M_{n_i}(D_i)$ where the $D_i$'s are artinian chain rings. By \cite[Theorem 4.1]{Dur0}, the subprojectivity and subinjectivity domains coincide over each $D_i$. But artinian chain rings are perfect so by Proposition \ref{cor-coinc-subflat-subproj} the subprojectivity and flat-precover completing domains coincide over each $D_i$. Thus, the flat-precover completing domains and subinjectivity domains coincide over each $D_i$. Then, by Proposition \ref{prop-coinc-morita}, the flat-precover completing domains and subinjectivity domains coincide over each $M_{n_i}(D_i)$ as it is Morita equivalent to $D_i$. We conclude by Proposition \ref{prop-prod-coinc} that the subinjectivity and the flat-precover completing domains coincide over $R$.
\end{proof}

It is worth noting that unlike the implication $2\Rightarrow 1$ in Theorem \ref{thm-coinc-subflat-subinj}, the implication $\In(M)\subseteq \F(M) \Rightarrow \F(M)=\In(M) $ for all modules $M$ does not hold in general. For instance, if we consider a von Neumann regular ring which is not semisimple then the inclusion $\In(M)\subseteq \F(M)$ holds for every module $M$, however there is a flat module $F$ which is not injective; that is, $\F(F)=R\text{-}\Mod$ but $\In(F)\neq R\text{-}\Mod$. It could be interesting to study rings which satisfy $\In(M)\subseteq \F(M)$ for any module $M$. We do not have a complete characterization of rings over which this property holds but we can observe that when this happens, $\I\subseteq \F(M)$ for any module $M$ and so $R$ is an $IF$-ring. We have the following result.

\begin{prop}\label{prop-embedds}
	The following conditions are equivalent for a module $M$:
	\begin{enumerate}
		\item $\I\subseteq \F(M)$.
		\item $E(M)\in \F(M)$.
		\item For any flat right $R$-module $F$, $F^+ \in  \F(M)$.
		\item $M$ embeds in a flat module.
		\item For any injective module $E$ and any submodule $N$ of $E$, if $N\in M^\perp$ then $E/N\in \F(M)$
	\end{enumerate}
\end{prop}

\begin{proof}
	$(1)\Leftrightarrow (2)\Leftrightarrow (3)\Leftrightarrow (4)$ Follow from Proposition \ref{prop-Inj-X(M)}.\\
	$(5)\Rightarrow (1)$ Take $N=0$.\\
	$(1)\Rightarrow (5)$ Let $E$ be an injective module. Consider the short exact sequence $0\to N \to E \to E/N\to 0$ with $N\in M^\perp$. Applying the functor $\Hom(M,-)$, we obtain the exact sequence $$\Hom(M,E)\to \Hom(M,E/N) \to \Ext^1(M,N) $$
	
	Since $\Ext^1(M,N)=0$, $\Hom(M,E)\to \Hom(M,E/N)$ is epic. Since $E\in \F(M)$, every map from $M$ to $E$ factors through a flat module and by Lemma \ref{Lemma-subX}, $E/N\in \F(M)$.
\end{proof}

Now, applying Proposition \ref{prop-embedds} to the class of finitely presented modules and to the class of pure-projective modules, we find characterizations of $IF$-rings in terms of flat-precover completing domains, some of which figure in \cite[Theorem 1]{Colby}.

We recall first that modules for which the subinjectivity domain is as small as possible are called indigent modules (see \cite[Definition 3.1]{PinSergio}). Here, modules for which the flat-precover completing domain is as small as possible will be called f-rugged modules to distinguish them from rugged modules defined in \cite{Rugged}. The flat-precover completing domain of such modules will consist of only flat modules.

Notice that f-rugged modules exist for any ring $R$. Indeed, consider $X=\oplus_{M\in S}M$, where $S$ is representative set of finitely presented modules. Then by Corollary \ref{cor-coinc-subflat-subproj} and \cite[Proposition 2.1]{Dur0} , we have $\F(X)=\f$. Recall that a module $N$ is said to be $FP$-injective if $\Ext^1(F,N)=0$, for every finitely presented module $F$.

\begin{cor}\label{cor-IF-rings}
	The following assertions are equivalent:
	\begin{enumerate}
		\item $R$ is an $IF$-ring.
		\item Any module embeds in a flat module.
		\item Any pure-projective module embeds in a flat module.
		\item Any finitely presented module embeds in a flat module.
		\item For any flat module $F$, $F\+$ is flat.
		\item For any injective module $E$ and any submodule $N$ of $E$, if $N$ is injective then $E/N$ is flat.
		\item For any injective module $E$ and any submodule $N$ of $E$, if $N$ is $FP$-injective then $E/N$ is flat.
		\item There exists an f-rugged module that embeds in a flat module.
		
	\end{enumerate}
	
\end{cor}

\begin{proof}
	$1 \Leftrightarrow 2 \Leftrightarrow 3 \Leftrightarrow 4 \Leftrightarrow 5 \Leftrightarrow 6 \Leftrightarrow 7$ Follow from Proposition \ref{prop-embedds} by considering each time every module $M\in R$-$\Mod$, every pure-projective module and every finitely presented module. Notice that a module $N$ is flat if and only $N\in \F(M)$ for every $M\in R$-$\Mod$  if and only $N\in \F(M)$ for every pure-projective module $M$ if and only $N\in \F(M)$ for every finitely presented module $M$.\\
	$1 \Rightarrow 8 $ Clear.\\
	$8 \Rightarrow 1 $ Suppose that there exists an f-rugged module $M$ which embeds in a flat module. Hence,  $\F(M)=\f$ and by Proposition \ref{prop-embedds}, $\I\subseteq \F(M)$. Whence, any injective module is flat so $R$ is an $IF$-ring.
\end{proof}

 The following proposition relates the existence of f-rugged modules to that of indigent modules. Notice that f-rugged modules are not necessarily finitely generated.

\begin{prop}\label{prop-frugged}
	Let $M$ be a module. If $R$ is right Noetherian and $M$ is f-rugged, then $M^+$ is indigent. If furthermore $M$ is finitely generated, then we have an equivalence.
\end{prop} 

\begin{proof}
	Let $N\in \In(M^+)$. By Proposition \ref{prop-Flatness-Subinj}, $N^+ \in \F(M)$. But $M$ is f-rugged so $N^+$ is flat and so $N$ is injective. We conclude that $M^+$ is indigent.\\
	Now let $M$ be a finitely generated such that $M^+$ is indigent and let us prove that $R$ is right Noetherian. We have $M^+$ is pure-injective and it is easy to see that the class of absolutely pure modules is inside the subinjectivity domain of any pure-injective module. Thus, every absolutely pure module is injective and by \cite[Theorem 3]{Megibben} we deduce that $R$ is right Noetherian.\\
	It is left to prove that $M$ is f-rugged. For let $N\in \F(M)$. Since $M$ is finitely generated and $R$ is right Noetherian, $M$ is finitely presented. By Proposition \ref{prop-Flatness-AbsPure2} we have $N^{++}\in \F(M)$ and applying Proposition \ref{prop-Flatness-Subinj} we deduce that $N^+\in \In(M^+)$. Since $M^+$ is indigent, $N^+$ is injective and thus $N$ is flat. We conclude that $M$ is f-rugged.
\end{proof}

Houda Amzil:   CeReMAR Center; Faculty of Sciences, Mohammed V University in Rabat, Rabat, Morocco.

\noindent e-mail address: houda.amzil@um5r.ac.ma; ha015@inlumine.ual.es

Driss Bennis:   CeReMAR Center; Faculty of Sciences, Mohammed V University in Rabat, Rabat, Morocco.

\noindent e-mail address: driss.bennis@um5.ac.ma; driss$\_$bennis@hotmail.com

J. R. Garc\'{\i}a Rozas: Departamento de  Matem\'{a}ticas,
Universidad de Almer\'{i}a, 04071 Almer\'{i}a, Spain.

\noindent e-mail address: jrgrozas@ual.es

Luis Oyonarte: Departamento de  Matem\'{a}ticas, Universidad de
Almer\'{i}a, 04071 Almer\'{i}a, Spain.

\noindent e-mail address: oyonarte@ual.es

\end{document}